\documentclass[11pt]{amsart}
\usepackage{amssymb, amstext, amscd, amsmath}
\usepackage[all]{xy}

\makeatletter
\def\@cite#1#2{{\m@th\upshape\bfseries%
[{#1\if@tempswa{\m@th\upshape\mdseries, #2}\fi}]}}
\makeatother
\theoremstyle{plain}
\newtheorem{theorem}{Theorem}[section]
\newtheorem{corollary}[theorem]{Corollary}
\newtheorem{proposition}[theorem]{Proposition}
\newtheorem{lemma}[theorem]{Lemma}

\theoremstyle{definition}
\newtheorem{definition}[theorem]{Definition}

\newtheorem{remark}[theorem]{Remark}
\theoremstyle{remark}

\newcommand{\bbC}{{\mathbb{C}}}

\newcommand{\bbN}{{\mathbb{N}}}

\newcommand{\A}{{\mathcal{A}}}

\newcommand{\C}{{\mathcal{C}}}

\newcommand{\F}{{\mathcal{F}}}

\renewcommand{\H}{{\mathcal{H}}}

\newcommand{\J}{{\mathcal{J}}}
\newcommand{\K}{{\mathcal{K}}}
\renewcommand{\L}{{\mathcal{L}}}
\newcommand{\M}{{\mathcal{M}}}

\renewcommand{\O}{{\mathcal{O}}}

\newcommand{\T}{{\mathcal{T}}}
\newcommand{\U}{{\mathcal{U}}}

\renewcommand{\phi}{\varphi}
\newcommand{\upchi}{{\raise.35ex\hbox{\ensuremath{\chi}}}}

\newcommand{\fin}{\operatorname{fin}}

\newcommand{\fn}{\operatorname{fin}}
\newcommand{\id}{{\operatorname{id}}}

\newcommand{\reg}{\operatorname{reg}}
\newcommand{\spn}{\operatorname{span}}
\newcommand{\sce}{\operatorname{sce}}

\newcommand{\ca}{\mathrm{C}^*}
\newcommand{\cenv}{\mathrm{C}^*_{\text{env}}}

\newcommand{\sca}[1]{\left\langle#1\right\rangle}

\begin{document}
\title[Hyperrigidity of Tensor algebras]{The hyperrigidity of tensor algebras of $\ca$-correspondences}

\author[E.G. Katsoulis]{Elias~G.~Katsoulis}
\address {Department of Mathematics 
\\East Carolina University\\ Greenville, NC 27858\\USA}
\email{katsoulise@ecu.edu}

\author[C. Ramsey]{Christopher~Ramsey}
\address {Department of Mathematics and Statistics
\\MacEwan University \\ Edmonton, AB \\Canada}
\email{ramseyc5@macewan.ca}

\begin{abstract}
Given a $\ca$-correspondence $X$, we give necessary and sufficient conditions for the tensor algebra $\T_X^+$ to be hyperrigid. In the case where $X$ is coming from a topological graph we obtain a complete characterization.
\end{abstract}


\maketitle

\section{Introduction}

A not necessarily unital operator algebra $\A$ is said to be \textit{hyperrigid} if given any non-degenerate $*$-homomorphism 
\[
\tau \colon \cenv(\A) \longrightarrow B(\H)
\]
then $\tau$ is the only completely positive, completely contractive extension of the restricted map  $\tau_{\mid \A}$.  Arveson coined the term hyperrigid in \cite{Arv} but he was not the only one considering properties similar to this at the time, e.g. \cite{Duncan}. 

There are many examples of hyperrigid operator algebras such as those which are Dirichlet but the situation was not very clear in the case of tensor algebras of C$^*$-correspondences. It was known that the tensor algebra of a row-finite graph is hyperrigid \cite{Duncan}, \cite{Kak} and Dor-On and Salmomon \cite{DorSal} showed that row-finiteness completely characterizes hyperrigidity for such graph correspondences. These approaches, while successful, did not lend themselves to a more general characterization.  

The authors, in a previous work \cite{KatRamHN}, developed a sufficient condition for hyperrigidity in tensor algebras. In particular, if Katsura's ideal acts non-degenerately on the left then the tensor algebra is hyperrigid. The motivation was to provide a large class of hyperrigid C$^*$-correspondence examples as crossed products of operator algebras behave in a very nice manner when the operator algebra is hyperrigid. This theory was in turn leveraged to provide a positive confirmation to the Hao-Ng isomorphism problem in the case of graph correspondences and arbitrary groups. For further reading on the subject please see \cite{KatRamMem, KatRamCP2, KatRamHN}.

In this paper, we provide a necessary condition for the hyperrigidity of a tensor algebra, that a C$^*$-correspondence cannot be $\sigma$-degenerate, and show that this completely characterizes the situation where the C$^*$-correspondence is coming from a topological graph, which generalizes both the graph correspondence case and the semicrossed product arising from a multivariable dynamical system.

\subsection{Regarding hyperrigidity}

The reader familiar with the literature recognizes that in our definition of hyperrigidity, we are essentially asking that the restriction on $\A$ of any non-degenerate representation of $\cenv(\A)$ possesses the \textit{ unique extension property} (abbr. UEP). According to \cite[Proposition 2.4]{DorSal} a representation $\rho : \A \rightarrow B(\H)$, degenerate or not, has the UEP if and only if $\rho$ is a maximal representation of $\A$, i.e., whenever $\pi$ is a representation of $\A$ dilating $\rho$, then $\pi = \rho \oplus \pi'$ for some representation $\pi'$. Our definition of hyperrigidity is in accordance with Arveson's nomenclature \cite{Arv}, our earlier work \cite{Kat, KatRamHN} and the works of Dor-On and Salomon \cite{DorSal} and Salomon \cite{Sal}, who systematized quite nicely the non-unital theory. 

An alternative definition of hyperrigidity for $\A$ may ask that \textit{any} representation of $\cenv(\A)$, not just the non-degenerate ones, possesses the UEP when restricted on $\A$. It turns out that for operator algebras with a positive contractive approximate unit\footnote{which includes all operator algebras appearing in this paper}, such a definition would be equivalent to ours \cite[Proposition 3.6 and Theorem 3.9]{Sal} . However when one moves beyond operator algebras with an approximate unit, there are examples to show that the two definitions differ. One such example is the non-unital operator algebra $\A_V$ generated by the unilateral forward shift $V$. It is easy to see that $\A_V$ is hyperrigid according to our definition and yet the zero map, as a representation on $\H=\bbC$, does not have the UEP. (See for instance \cite[Example 3.4]{Sal}.)

\section{Main results}

A $\ca$-correspondence $(X,\C,\phi_X)$ (often just $(X,\C)$) consists of a $\ca$-algebra $\C$, a Hilbert $\C$-module $(X, \sca{\phantom{,},\phantom{,}})$ and a
(non-degenerate) $*$-homomorphism $\phi_X\colon \C \rightarrow \L(X)$ into the C$^*$-algebra of adjointable operators on $X$. 

An isometric (Toeplitz) representation $(\rho,t, \H)$ of a $\ca$-correspondence $(X,\C)$ consists of  a non-degenerate $*$-homomorphism $\rho\colon \C \rightarrow B(\H)$ and a linear map $t\colon X \rightarrow B(\H)$, such that 
 $$\rho(c)t(x)=t(\phi_X(c)(x)), \ \ \textrm{and}$$
 $$t(x)^*t(x')=\rho(\sca{x,x'}),$$
for all $c\in \C$ and $x, x'\in X$. 
These relations imply that the $\ca$-algebra generated by this isometric representation equals the closed linear span of $$t(x_1)\cdots  t(x_n)t(y_1)^*\cdots t(y_m)^*, \quad x_i,y_j\in X.$$
Moreover, there exists a $*$-homomorphism
$\psi_t:\K(X)\rightarrow B$, such that $$\psi_t(\theta_{x,y})=
t(x)t(y)^*,$$ where $\K(X) \subset \L(X)$ is the subalgebra generated by the operators $\theta_{x,y}(z) = x \langle y,z\rangle, \ x,y,x\in X$, which are called by analogy the compact operators.

The Cuntz-Pimsner-Toeplitz $\ca$-algebra $\T_X$ is defined as the $\ca$-algebra generated by the image of $(\rho_{\infty} , t_{\infty})$, the universal isometric representation. This is universal in the sense that for any other isometric representation there is a $*$-homomorphism of $\T_X$ onto the $\ca$-algebra generated by this representation in the most natural way.

The \emph{tensor algebra} $\T_{X}^+$ of a $\ca$-correspondence
$(X,\C)$ is the norm-closed subalgebra of $\T_X$ generated by $\rho_{\infty}(\C)$ and $t_{\infty}(X)$. See \cite{MS} for more on these constructions.
 
Consider Katsura's ideal $$\J_X\equiv \ker\phi_X^{\perp}\cap \phi_X^{-1}(\K(X)).$$ An isometric representation $(\rho, t)$ of $(X, \C,\phi_X)$ is said to be covariant (or Cuntz-Pimsner) if and only if $$\psi_t ( \phi_X(c)) = \rho (c),$$ for all $c \in \J_X$. The Cuntz-Pimsner algebra $\O_X$ is the universal $\ca$-algebra for all isometric covariant representations of $(X, \C)$, see \cite{KatsuraJFA} for further details. Furthermore, the first author and Kribs \cite[Lemma 3.5]{KatsoulisKribsJFA} showed that $\O_X$ contains a completely isometric copy of $\T_X^+$ and $\cenv(\T_X^+)\simeq \O_X$.

We turn now to the hyperrigidity of tensor algebras. In \cite{KatRamHN} a sufficient condition for hyperrigidity was developed, Katsura's ideal acting non-degenerately on the left of $X$. To be clear, non-degeneracy here means that $\overline{\phi_X(\J_X)X} = X$ which by Cohen's factorization theorem implies that we actually have $\phi_X(\J_X)X = X$.

\begin{theorem}[Theorem 3.1, \cite{KatRamHN}]\label{thm:hyperrigid}
Let $(X, \C)$ be a $\ca$-correspondence with $X$ countably generated as a right Hilbert $\C$-module. If $\phi_X(\J_X)$ acts non-degenerately on $X$, then $\T^+_X$ is a hyperrigid operator algebra.
\end{theorem}

The proof shows that if $\tau'\colon \O_X  \longrightarrow B(\H)$ is a completely contractive and completely positive map that agrees with a $*$-homomorphism of $\O_X$ on $\T^+_{X}$ then the multiplicative domain of $\tau'$ must be everything. This is accomplished through the multiplicative domain arguments of \cite[Proposition 1.5.7]{BRO} and the fact that by $X$ being countably generated, Kasparov's Stabilization Theorem implies the existence of a sequence $\{x_n \}_{n=1}^{\infty}$ in $X$ so that $ \sum_{n=1}^{k} \theta_{x_n, x_n}$, $k=1, 2, \dots$, is an approximate unit for $\K(X)$. After quite a lot of inequality calculations one arrives at the fact that all of $\T_X^+$ is in the multiplicative domain and thus so is $\O_X$.

A $\ca$-correspondence $(X, \C)$ is called \textit{regular} if and only if $\C$ acts faithfully on $X$ by compact operators, i.e., $\J_X= \C$. We thus obtain the following which also appeared in \cite{KatRamHN}.
\begin{corollary} \label{regular hyper}
The tensor algebra of a regular, countably generated $\ca$-correspondence is necessarily hyperrigid.
\end{corollary}

We seek a converse to Theorem~\ref{thm:hyperrigid}.

\begin{definition}
Let $(X, \C)$ be a $\ca$-correspondence and let $\J_X$ be Katsura's ideal. We say that $\phi_X(\J_X)$ acts \textit{$\sigma$-degenerately} on $X$ if there exists a representation $\sigma \colon \C \rightarrow B(\H)$ so that 
\[
\phi_X(\J_X)X\otimes_{\sigma}\H \neq X\otimes_{\sigma} \H.
\]
\end{definition}

\begin{remark} In particular, if there exists $n \in \bbN$ so that 
\[
(\phi_X(\J_X)\otimes \id)X^{\otimes n }\otimes_{\sigma}\H \neq X^{\otimes n}\otimes_{\sigma} \H.
\]
then by considering the Hilbert space $\K :=X^{\otimes n-1 }\otimes_{\sigma}\H$, we see that 
\[
\phi_X(\J_X)X\otimes_{\sigma}\K \neq X\otimes_{\sigma} \K.
\]
and so $\phi_X(\J_X)$ acts $\sigma$-degenerately on $X$.
\end{remark}

The following gives a quick example of a $\sigma$-degenerate action. Note that this is possibly stronger than having a not non-degenerate action.
\begin{proposition}
Let $(X, \C)$ be a $\ca$-correspondence. If $(\phi_X(\J_X)X)^{\perp}\neq \{0 \}$, then $\phi_X(\J_X)$ acts $\sigma$-degenerately on $X$.
\end{proposition}
\begin{proof}
Let $0 \neq f \in (\phi_X(\J_X)X)^{\perp}$. Let $\sigma : \C\rightarrow B(\H)$ be a $*$-representation and $h \in \H$ so that $\sigma \big( \langle f , f \rangle^{1/2}\big)h \neq 0$.
Then,
\[
\langle f\otimes_{\sigma}h ,  f\otimes_{\sigma}h \rangle= \langle h , \sigma (( \langle f , f \rangle)h \rangle = \| \sigma \big( \langle f , f \rangle^{1/2}\big)h\|\neq0.
\]
A similar calculation shows that $$0 \neq f\otimes_{\sigma}h  \in (\phi_X(\J_X)X \otimes_{\sigma}\H)^{\perp}$$ and we are done.
\end{proof}

We need the following 

\begin{lemma} \label{l;use1}
Let $(X, \C)$ be a $\ca$-correspondence and $(\rho, t)$ an isometric representation of $(X, \C)$ on $\H$. 
\begin{enumerate}
\item If $\M\subseteq \H$ is an invariant subspace for $(\rho\rtimes t)(\T_X^+)$, then the restriction $(\rho_{\mid_{\M}}, t_{\mid_{\M}})$ of $(\rho, t)$ on $\M$ is an isometric representation.

\item If $\rho(c)h = \psi_t(\phi_X(c))h,$ for all $c \in \J_X$ and $h \in [t(X)\H]^{\perp}$, then $(\rho, t)$ is a Cuntz-Pimsner representation.
 \end{enumerate}
\end{lemma}

\begin{proof}
(i) If $p$ is the orthogonal projection on $\M$, then $p$ commutes with $\rho(\C)$ and so $\rho_{\mid_{\M}}(\cdot)= p\rho (\cdot) p$ is a $*$-representation of $\C$. 

Furthermore, for $x, y \in X$, we have
\begin{align*}
t_{\mid_{\M}}(x)^*t_{\mid_{\M}}(y) &= pt(x)^*pt(y)p \\
&=pt(x)^*t(y)p \\
             &= p\rho(\langle x, y\rangle ) p= \rho_{\mid_{\M}}(\langle x, y\rangle)
             \end{align*}
             and the conclusion follows.

(ii) It is easy to see on rank-one operators and therefore by linearity and continuity on all compact operators $K \in \K(X)$ that 
\[
t(Kx)=\psi_t(K)t(x), \quad x \in X.
\]
Now if $c \in \J_X$, then for any $x \in X$ and $ h \in \H$ we have
\[
\rho(c)t(x)h= t(\phi_X(c)x)h =\psi_t(\phi_X(c))t(x)h.
\]            
By assumption $ \rho(c)h = \psi_t(\phi_X(c))h$, for any $h \in [t(X)\H]^{\perp}$ and the conclusion follows.  \end{proof}

\begin{theorem} \label{thm;converse}
Let $(X, \C)$ be a $\ca$-correspondence. If Katsura's ideal $\J_X$ acts $\sigma$-degenerately on $X$ then the tensor algebra $\T^+_X$ is not hyperrigid.
\end{theorem}

\begin{proof}
Let $\sigma \colon \C \rightarrow B(\H)$ so that 
\[
\phi_X(\J_X)X\otimes_{\sigma}\H \neq X\otimes_{\sigma} \H
\]
and let $ \M_0:= (\phi_X(\J_X)X\otimes_{\sigma}\H)^{\perp}$.

We claim that
\begin{equation} \label{eq;use1}
(\phi_X(\J_X)\otimes I)\M_0=\{ 0\}.
\end{equation}
Indeed for any $f \in \M_0$ and $j \in \J_X$ we have 
\begin{align*}
\big \langle (\phi_X(j)\otimes I)f \, , \, (\phi_X(j)\otimes I)f \big \rangle =\langle f, (\phi_X(j^*j)\otimes I) f\rangle = 0
\end{align*}
since $f \in (\phi_X(\J_X)X\otimes_{\sigma}\H)^{\perp}$. This proves the claim.

We also claim that 
\begin{equation} \label{eq;use2}
(\phi_X(\C)\otimes I)\M_0 = \M_0.
\end{equation}
Indeed this follows from the fact that 
\[(\phi_X(\C)\otimes I)(\phi_X(\J_X)X\otimes_{\sigma}\H) = \phi_X(\J_X)X\otimes_{\sigma}\H,
\]
which is easily verified.

Using the subspace $\M_0$ we produce a Cuntz-Pimsner representation $(\rho, t)$ of $(X, \C)$ as follows. Let $(\rho_{\infty}, t_{\infty})$ be the universal representation of $(X, \C)$ on the Fock space $\F(X)=\oplus_{n =0}^{\infty} X^{\otimes n}$, $X^{\otimes 0}:=\C$. Let 
\begin{align*}
\rho_0 &\colon \C \longrightarrow B(\F(X)\otimes_{\sigma} \H );  c \longmapsto \rho_{\infty}(c)\otimes I \\
t_0 &\colon X \longrightarrow B(\F(X)\otimes_{\sigma} \H ) ;  x \longmapsto t_{\infty}(x)\otimes I.
\end{align*}
Define
\begin{align*}
\M:&= 0\oplus \M_0 \oplus (X\otimes \M_0) \oplus (X^{\otimes 2}\otimes  \M_0)  \oplus\dots \\
&=(\rho_0\rtimes t_0)(\T_X^+)(0 \oplus\M_0 \oplus 0 \oplus 0\oplus \dots )\subseteq \F(X)\otimes_{\sigma}\H,
\end{align*}
with the second equality following from (\ref{eq;use2}). Clearly, $\M$ is an invariant subspace for $(\rho_0\rtimes t_0)(\T_X^+)$. 

Let $\rho:={\rho_0}_{\mid_{\M}}$ and $t:={t_0}_{\mid_{\M}}$. By Lemma~\ref{l;use1}(i), $(\rho, t)$ is a representation of $(X, \C)$. We claim that $(\rho, t)$ is actually Cuntz-Pimsner.

Indeed by Lemma~\ref{l;use1}(ii) it suffices to examine whether $\psi_t(\phi_X(j))h = \rho(j)h$, for any $h \in \M \ominus t(X)\M$. Note that since
\[
t(X)\M= 0 \oplus 0 \oplus (X\otimes \M_0 ) \oplus  (X^{\otimes 2}\otimes  \M_0)  \oplus ... ,
\]
we have that 
\[
\M \ominus t(X)\M = 0 \oplus \M_0 \oplus 0 \oplus  0\oplus \dots.
\]
From this it follows that for any $h \in \M \ominus t(X)\M$ we have 
\[
t_0(x)^*h \in  (\C\otimes_{\sigma}\H) \oplus 0 \oplus 0\oplus ... ,\quad x \in X
\]
 and so in particular for any $j \in \J_X$ we obtain
 $$\psi_t(\phi_X(j))h  \in {t_0}_{\mid_{\M}}(X)({t_0}_{\mid_{\M}})(X)^*h = \{0\}.$$
 On the other hand, 
 \[
 \rho(j)h \in 0\oplus (\phi_X(\J_X)\otimes I)\M_0\oplus 0\oplus 0\oplus\dots =\{ 0\},
 \]
 because of (\ref{eq;use2}). Hence $(\rho, t)$ is Cuntz-Pimsner.
 
 At this point by restricting on $\T_X^+$, we produce the representation $\rho \rtimes t \mid_{\T_X^+}$ of $\T_X^+$ coming from a $*$-representation of its $\ca$-envelope $\O_X$, which admits a dilation, namely $\rho_0 \rtimes t_0 \mid_{\T_X^+}$. If we show now that $\rho_0 \rtimes t_0 \mid_{\T_X^+}$ is a non-trivial dilation of $\rho \rtimes t \mid_{\T_X^+}$, i.e. $\M_0$ is not reducing for $(\rho_0\rtimes t_0)(\T_X^+)$, then $\rho \rtimes t \mid_{\T_X^+}$ is not a maximal representation of $\T^+_X$.  Proposition 2.4~\cite{DorSal} shows $\rho \rtimes t \mid_{\T_X^+}$ does not have the UEP and so $\T_X^+$ is not hyperrigid, as desired.
 
 Towards this end, note that 
\[
\M^{\perp} = \C \oplus (\phi_X(\J_X)X\otimes_{\sigma}\\H )\oplus (X\otimes \M_0 )^{\perp}\oplus \dots
\]
and so
\begin{align*}
t_0(X)\M^{\perp}=0\oplus (X \C\otimes_{\sigma}\H)\oplus 0 \oplus0 \oplus \dots
 \nsubseteq \M^{\perp}
\end{align*}
Therefore $\M^{\perp}$ is not an invariant subspace for $(\rho_0\rtimes t_0)(\T_X^+)$ and so $\M$ is not a reducing subspace for $(\rho_0\rtimes t_0)(\T_X^+)$. This completes the proof.
 \end{proof}

\section{Topological graphs}

A broad class of $\ca$-correspondences arises naturally from the concept of a topological graph. For us, a topological graph $G= (G^0, G^1, r , s)$ consists of two $\sigma$-locally compact spaces $G^0$, $G^1$, a continuous proper map $r: G^1 \rightarrow G^0$ and a local homeomorphism $s: G^1 \rightarrow G^0$. The set $G^0$ is called the base (vertex) space and $G^1$ the edge space. When $G^0$ and $G^1$ are both equipped with the discrete topology, we have a discrete countable graph.

With a given topological graph $G= (G^0, G^1, r , s)$ we associate a $\ca$-correspondence $X_{G}$ over $C_0(G^0)$. The right and left actions of $C_0(G^0 )$ on $C_c ( G^1)$ are given by
\[
(fFg)(e)= f(r(e))F(e)g(s(e))
\]
for $F\in C_c (G^1)$, $f, g \in C_0(G^0)$ and $e \in G^1$. The inner product is defined for $F, H \in C_c ( G^1)$ by
\[
\left< F \, | \, H\right>(v)= \sum_{e \in s^{-1} (v)} \overline{F(e)}H(e)
\]
for $v \in G^0$. Finally, $X_{G}$ denotes the completion of $C_c ( G^1)$ with respect to the norm
\begin{equation} \label{norm}
\|F\| = \sup_{v \in G^0} \left< F \, | \, F\right>(v) ^{1/2}.
\end{equation}

When $G^0$ and $G^1$ are both equipped with the discrete topology, then the tensor algebra $\T_{G}^+ \equiv \T^+_{X_{G}}$ associated with $G$ coincides with the quiver algebra of Muhly and Solel \cite{MS}. See \cite{Raeburn} for further reading.

Given a topological graph  $G= (G^0, G^1, r, s)$, we can describe the ideal $\J_{X_G}$ as follows. Let 
\begin{align*}
G^0_{\sce}&= \{v \in G^0\mid v \mbox{ has a neighborhood } V \mbox{ such that } r^{-1}(V) = \emptyset \} \\
G^0_{\fin}&= \{v \in G^0\mid v \mbox{ has a neighborhood } V \mbox{ such that } r^{-1}(V) \mbox{ is compact} \} 
\end{align*}  
Both sets are easily seen to be open and in \cite[Proposition 1.24]{Katsura} Katsura shows that 
\[
\ker\phi_{X_G}= C_0 (G^0_{\sce})\,\, \mbox{ and }\,\, \phi_{X_G}^{-1}(\K(X_G)) =C_0(G^0_{\fn}).
\]
From the above it is easy to see that $\J_{X_G}=C_0(G^0_{\reg})$, where 
\[
G^0_{\reg}:= G^0_{\fin} \backslash \overline{G^0_{\sce}}.
\]

We need the following 
\begin{lemma} \label{l;toppled}
Let $G= (G^0, G^1, r, s)$ be a topological graph. Then $r^{-1}\big( G^0_{\reg} \big)=G^1$ if and only if $r : G^1 \rightarrow G^0$ is a proper map satisfying $r(G^1) \subseteq \big(\overline{r(G^1)}\big)^{\circ}$.
\end{lemma}

\begin{proof}Notice that 
\[
r^{-1}(G^0_{\reg})=r^{-1}(G^0_{\fn}) \cap r^{-1}(\overline{G^0_{\sce}})^c
\]
and so $r^{-1}\big( G^0_{\reg} \big)=G^1$ is equivalent to $r^{-1}(G^0_{\fn}) =r^{-1}(\overline{G^0_{\sce}})=G^1$

First we claim that $r^{-1}(G^0_{\fn}) = G^1$ if and only if $r$ is a proper map.
Indeed, assume that $r^{-1}(G^0_{\fn}) = G^1$ and let $K\subseteq r(G^1)$ compact in the relative topology. For every $x \in K$, let $V_x$ be a compact neighborhood of $x$ such that $r^{-1}(V_x)$ is compact and so $r^{-1}(V_x \cap K))$ is also compact. By compactness, there exist $x_1, x_2, \dots , x_n \in K$ so that $K =  \cup_{i=1}^{n}(V_{x_i} \cap K)$ and so $$r^{-1} (K) =  \cup_{i=1}^{n}r^{-1}(V_{x_i} \cap K)$$ and so $r^{-1} (K)$ is compact.

Conversely, if $r$ is proper then any compact neighborhood V of any point in $G^0$ is inverted by $r^{-1}$ to a compact set and so $r^{-1}(G^0_{\fn}) = G^1$.

We now claim that $r^{-1}(\overline{G^0_{\sce}})=\emptyset$ if and only $r(G^1) \subseteq \big(\overline{r(G^1)}\big)^{\circ}$.

Indeed, $e \in r^{-1}(\overline{G^0_{\sce}})$ is equivalent to $r(e) \in \overline{(r(G^1)^c)^{\circ}}$ and so 
$r^{-1}(\overline{G^0_{\sce}})=\emptyset$ is equivalent to 
\[
r(G^1)\subseteq \left(    \overline{(r(G^1)^c)^{\circ}}    \right)^c =\big(\overline{r(G^1)}\big)^{\circ},
\]
as desired.
\end{proof}

If $G= (G^0, G^1, r, s)$ is a topological graph and  $S \subseteq G^1 $, then $N(S)$ denotes the collection of continuous functions $F \in X_{G}$ with $F_{|S}=0$, i.e., vanishing at $S$. The following appears as Lemma~4.3(ii) in \cite{KatsLoc}.

 \begin{lemma} \label{topology}
Let $G= (G^0, G^1, r, s)$ be a topological graph.
If $S_1 \subseteq G^0$, $S_2 \subseteq G^1$ closed, then
 \[
 N(r^{-1}(S_1) \cup S_2)=\overline{\spn} \{(f\circ r) F\mid f_{|S_1} = 0, F_{|S_2} = 0\}
 \]
 \end{lemma}

\begin{theorem} Let $G =(G^0, G^1, r, s)$ be a topological graph and let $X_G$ the $\ca$-correspondence associated with $G$. Then the following are equivalent
\begin{itemize}
\item[(i)] the tensor algebra $\T_{X_G}^+$ is hyperrigid 
\item[(ii)] $\phi(\J_{X_G})$ acts non-degenerately on $X_G$
\item[(iii)] $r: G^1 \rightarrow G^0$ is a proper map satisfying $ r(G^1) \subseteq \big(\overline{r(G^1)}\big)^{\circ}$
\end{itemize}
\end{theorem}

\begin{proof}
If $\phi(\J_{X_G})$ acts non-degenerately on $X_G$, then Theorem~\ref{thm:hyperrigid} shows that $\T_{X_G}^+$ is hyperrigid. Thus (ii) implies (i).

For the converse, assume that $\phi(\J_{X_G})$ acts degenerately on $X_G$. If we verify that $\phi(\J_{X_G})$ acts $\sigma$-degenerately on $X_G$, then Theorem~\ref{thm;converse} shows that   $\T_{X_G}^+$ is not hyperrigid and so (i) implies (ii).

Towards this end note that $\J_{X_G} = \C_0(\U)$ for some proper open set $\U\subseteq G^0$. (Actually we know that $\U = G^0_{\reg}$ but this is not really needed for this part of the proof!) Hence
\begin{align} \label{eq;iii}
\begin{split}
\phi(\J_{X_G}) X_G &=\overline{\spn} \{(f\circ r)F\mid f_{\mid \U^c}=0 \} \\
&=N(r^{-1}(\U)^c),
\end{split}
\end{align}
according to Lemma~\ref{topology}.

Since $\phi(\J_{X_G})$ acts degenerately on $X_G$, (\ref{eq;iii}) shows that $r^{-1}(\U)^c \neq \emptyset$. Let $e \in r^{-1}(\U)^c$ and let $F\in C_c(G^1)\subseteq X_{G}$ with $F(e)=1$ and $F(e')=0$, for any other $e'\in G^1$ with $s(e')=s(e)$. Consider the one dimensional representation $\sigma: C_0(G_0) \rightarrow \bbC$ coming from evaluation at $s(e)$. We claim that 
\[
\phi_{X_G}(\J_{X_G})X_G \otimes_{\sigma}\bbC\neq X_G\otimes_{\sigma}\bbC.
\]
Indeed for any $G \in \phi(\J_{X_G}) X_G = N(r^{-1}(\U)^c)$ we have 
\begin{align*}
\langle F\otimes_{\sigma}1, G\otimes_{\sigma}1 \rangle &=\langle 1, \sigma(\langle F,G\rangle 1)= \langle F,G\rangle s(e)\\
							&=\sum_{s(e')=s(e)} \overline{F(e')}G(e')\\
							&= \overline{F(e)}G(e) = 0.
\end{align*}
Furthermore,
\[
\langle F\otimes_{\sigma}1, F\otimes_{\sigma}1 \rangle s(e)= |F(e)|^2 =1
\]
and so $0 \neq F\otimes_{\sigma}1 \in (\phi_{X_G}(\J_{X_G})X_G \otimes_{\sigma}\bbC)^{\perp}$. This establishes the claim and finishes the proof of (i) implies (ii).

Finally we need to show that (ii) is equivalent to (iii). Notice that (\ref{eq;iii}) implies that $\phi(\J_{X_G})$ acts degenerately on $X_G$ if and only if
\[
r^{-1}(\U)^c=r^{-1}(G^0_{\reg})^c=\emptyset.
\]
The conclusion now follows from Lemma~\ref{l;toppled}.
\end{proof}

The statement of the previous Theorem takes its most pleasing form when $G^0$ is a compact space. In that case $\T^+_X$ is hyperrigid if and only is $G^1$ is compact and $r(G^1)\subseteq G^0$ is clopen.


\vspace{0.1in}

{\noindent}{\it Acknowledgement.} Both authors would like to thank MacEwan University for providing project funding to bring the first author out to Edmonton for a research visit. The first author received support for this project in the form of a Summer Research Award from the Thomas Harriott College of Arts Sciences at ECU. He also expresses his gratitude to Adam Dor-On and Guy Salomon for several discussions regarding their work on hyperrigidity. The second author was partially supported by an NSERC grant.


\end{document}